\documentclass[12pt]{amsart}
\usepackage{verbatim}
\usepackage{graphicx}
\usepackage{amssymb}
\usepackage{amsfonts}
\usepackage{amsmath}
\usepackage{amsthm}
\usepackage{hyperref}

\def\RR{\mathbb{R}}

\def\NN{\mathbb{N}}
\def\ZZ{\mathbb{Z}}

\def\eps{\varepsilon}
\DeclareMathOperator{\tr}{Tr}

\DeclareMathOperator{\pos}{Pos}
\def\rx{\mathbb{R}[\underline{x}]}
\DeclareMathOperator{\sym}{S}
\DeclareMathOperator{\mat}{M}
\def\sn{\sym_n(\rx)}
\def\mn{\mat_n(\rx)}
\def\snn{\sym_{n+1}(\rx)}
\def\mnn{\mat_{n+1}(\rx)}
\def\sos{\sum \rx^2}
\def\sohs{\sum \mn^2}
\def\sohss{\sum \mnn^2}
\def\G{\mathcal{G}}

\newtheorem{theorem}{Theorem}[section]
\newtheorem*{citethm}{Theorem}
\newtheorem{lemma}[theorem]{Lemma}
\newtheorem{prop}[theorem]{Proposition}
\newtheorem{cor}[theorem]{Corollary}

\newtheorem{conj}{Question}
\renewcommand{\theconj}{\Alph{conj}}
\theoremstyle{definition}
\newtheorem{remark}[theorem]{Remark}
\newtheorem{ex}[theorem]{Example}

\begin{document}
\title{Finsler's Lemma for Matrix Polynomials}

\author{Jaka Cimpri\v c}
\address{University of Ljubljana, Faculty of Mathematics and Physics, Department of Mathematics, Jadranska 21, SI-1000 Ljubljana, Slovenija}
\email{Jaka.Cimpric@fmf.uni-lj.si}
\thanks{Research supported by the grant P1--0222 from the Slovenian Research Agency}

\subjclass[2000]{15A54, 14P99, 13J25, 06F25}
\keywords{matrix polynomials, real algebraic geometry}

\begin{abstract}
Finsler's Lemma charactrizes all pairs of symmetric $n \times n$ real matrices $A$ and $B$ 
which satisfy the property that $v^T A v>0$ for every nonzero $v \in \RR^n$ such that $v^T B v=0$.
We extend this characterization to all symmetric matrices of real multivariate polynomials, 
but we need an additional assumption that $B$ is negative semidefinite outside some ball.
We also give two applications of this result to Noncommutative Real Algebraic Geometry
which for $n=1$ reduce to the usual characterizations of positive polynomials on varieties and
on compact sets.
\end{abstract}

\maketitle

\section{Introduction}
\label{sec1}

The aim of this paper is to generalize the following result from matrices to matrix polynomials.

\begin{lemma}[Finsler 1937]
\label{finsler1}
Suppose that $F$ and $G$ are symmetric $n \times n$ real matrices such that
for every nonzero $v \in \RR^n$ which satisfies $v^T G v = 0$
we have that $v^T F v >0$. Then there exists a real number $r$ such that 
$F-r G$ is positive definite. (The converse is clear.)
\end{lemma}

A symmetric $n \times n$ real matrix $A$ is positive definite 
(resp. positive semidefinite) if $v^T A v>0$ (resp. $v^T A v \ge 0$) for
every nonzero $v \in \RR^n$. In this case we write $A \succ 0$ (resp. $A \succeq 0$.)
We will also discuss the following variant of Lemma \ref{finsler1}.

\begin{lemma}
\label{finsler2}
Suppose that $F$ and $G$ are symmetric $n \times n$ real matrices such that
for every nonzero $v \in \RR^n$ which satisfies $v^T G v \ge 0$
we have that $v^T F v >0$. Then there exists a real number $r>0$ 
such that $F-r G$ is positive definite. (The converse is clear and it works for $r \ge 0$.)
\end{lemma}

\begin{proof} Let us show that the $(n+1) \times (n+1)$ matrices
$$\tilde{F}:=\left[ \begin{array}{cc} F & 0 \\ 0 & 0 \end{array} \right]
\quad \text{ and } \quad
\tilde{G}:=\left[ \begin{array}{cc} G & 0 \\ 0 & -1 \end{array} \right]$$
satisfy the assumptions of Lemma \ref{finsler1}. Pick 
$\tilde{v}=(v,\alpha) \in \RR^{n+1}$ such that $\tilde{v}^T \tilde{G} \tilde{v}=0$.
Now $v^T G v=\alpha^2 \ge 0$ implies that $\tilde{v}^T \tilde{F} \tilde{v} = v^T F v>0$ by assumption.
By Lemma \ref{finsler1} there exists $r \in \RR$ such that
$\tilde{F}-r \tilde{G} \succ 0$. It follows that $r>0$ and $F-rG \succ 0$.
\end{proof}

Let us recall the usual notation for polynomials and matrix polynomials.
We will write $\rx:=\RR[x_1,\ldots,x_d]$ for the algebra of all real polynomials in $d$ variables,
$\sos$ for the set of all finite sums of squares of polynomials from $\rx$,
$\mn$ for the algebra of all $n \times n$ matrix polynomials,
$\sn$ for the real vector space of all symmetric $n \times n$ matrix polynomials
and $\sum \mn^2$ for the set of all finite sums of expressions of the form $H^T H$ where $H \in \mn$.
Such expressions are called \textit{hermitian squares} of  matrix polynomials. 
Note that the identity matrix $I_n$ is a hermitian square.
For every subset $K \subseteq \RR^d$ we write $\pos(K):=\{f \in \rx \mid f|_K \ge 0\}$.
For every $A\subset \rx$ and $B \subset \sn$ we write $A \cdot B:=\{\sum_i a_i b_i \mid a_i \in A,b_i \in B\}$ (finite sums).

We will discuss the following questions which can be considered as analogues
of Lemmas \ref{finsler1} and \ref{finsler2} for matrix polynomials.

\begin{conj}
\label{conj1}
For which $F,G \in \sn$ are the following equivalent:
\begin{enumerate}
\item[(A1)] For every $a \in \RR^d$ and every nonzero $v \in \RR^n$ such that $v^T G(a) v = 0$ we have that $v^T F(a) v>0$.
\item[(A2)] There exists $s \in \sos$ such that $(1+s)F  \in I_n+O_G$ where $O_G:=\sum \mn^2+\rx \cdot G.$.
\item[(A3)] There exists $s \in \sos$ such that $(1+s)F \in I_n+\pos(L_G) \cdot O_G$ where 
$L_G:=\{a \in \RR^d \mid v^T G(a) v =0 \text{ for some } 0 \ne v \in \RR^n\}.$
\end{enumerate}
\end{conj}

\begin{conj}
\label{conj2}
For which $F,G \in \sn$ are the following equivalent:
\begin{enumerate}
\item[(B1)] For every $a \in \RR^d$ and every nonzero $v \in \RR^n$ such that $v^T G(a) v \ge 0$ we have that $v^T F(a) v>0$.
\item[(B2)] There exists $s \in \sos$ such that $(1+s)F \in I_n+N_G$ where $N_G:=\sum \mn^2+(\sos) \cdot G$.
\item[(B3)] There exists $s \in \sos$ such that $(1+s)F \in I_n+\pos(K_G) \cdot N_G$ where 
$K_G:=\{a \in \RR^d \mid v^T G(a) v \ge 0 \text{ for some } 0 \ne v \in \RR^n\}.$
\end{enumerate}
\end{conj}

The motivation for studying Question \ref{conj2} comes from (one version of) the Noncommutative Real Algebraic Geometry
for matrix polynomials.
The question there is the following: \textit{For given $G \in \sn$ characterize all $F \in \sn$
which satisfy (B1).} Question \ref{conj2} suggests such a characterization in terms of the sets
$N_G$ and $\pos(K_G) \cdot N_G$ which can be considered as noncommutative analogues of
quadratic modules and preorderings respectively. The aim of this paper is to show that this
characterization does not work for every $G \in \sn$ (see Example \ref{exmain}), but it works for 
those $G \in \sn$ that are negative semidefinite outside some ball (see Proposition \ref{main}.)
This rather restrictive assumption is satisfied for example
if $G  =-\sum_{i=1}^m G_i^T G_i$ for some $G_1,\ldots,G_m \in \mn$ or if $K_G$ is compact, when we get
the following results as corollaries (see Theorems \ref{mainth1} and \ref{mainth2}).
For both results, the opposite direction is clear.

\begin{citethm}[Positivestellensatz for varieties]
Suppose that $G_1,\ldots,G_m\in\mn$ and write $J$ for the left ideal generated by them.
For every $F \in \sn$ which satisfies $v^T F(a) v >0$ for every $a \in \RR^d$ and every
nonzero $v \in \RR^n$ such that $G_1(a)v=\ldots=G_m(a)v=0$, there exists $s \in \sos$
such that $(1+s)F \in I_n+\sum \mn^2+J+J^T$.
\end{citethm}

\begin{citethm}[Compact Positivstellensatz with one constraint]
Suppose  that the set $K_G:=\{a \in \RR^d \mid v^T G(a) v \ge 0 \text{ for some } 0 \ne v \in \RR^n\}$ is compact
for some $G \in \sn$. 
Then for every $F \in \sn$ which satisfies $v^T F(a)v>0$ for every $a \in \RR^d$ and every nonzero $v \in \RR^n$ 
such that $v^T G(a) v \ge 0$, there exists $\eps>0$ such that $F -\eps I_n \in N_G$.
\end{citethm}

The Positivstellensatz for varieties is related to the one-sided Real Nullstellenatz from \cite{c3}.
Similar results also exist for free polynomials, see \cite{h1}, \cite{h2}, \cite{klep}.
(One-sided Real nullstellensatz for free polynomials is discussed in \cite{chmn},\cite{chkmn},\cite{nelson}.)
The Compact Positivstellenatz is related to Theorem 2.1. in \cite{c1} which can be considered as an analogous
version of the Archimedean Positivstellensatz with finitely many constraints.

Let us explain the organization of the paper.
In Section \ref{sec2} we give geometric reformulations of Questions \ref{conj1} and \ref{conj2} 
that are easier to work with. We also show that the results about Question \ref{conj2} in dimension $n$
follow from the results about Question \ref{conj1} in dimension $n+1$ as can be expected from the proof of Lemma \ref{finsler2}.

In Section \ref{sec3} we prove our main technical result, Proposition \ref{main}, which says that
the equivalences in Questions \ref{conj1} and \ref{conj2} hold for every $F \in \sn$ 
and every $G \in \sn$ which is negative semidefinite outside some ball.
We also give asymptotic reformulations of Questions \ref{conj1} and \ref{conj2}.
In Section \ref{sec4} we deduce from Proposition \ref{main} the above\-mentioned Positivstellens\" atze
for varieties and for compact $K_G$. 

In Section \ref{sec5} we introduce the notion of a weak preordering and show that the set $\pos(K_G) \cdot N_G$
is a weak preordering but it need not the smallest weak preordering which contains $G$. We also show that
for $n=2$ and $d=1$, (B1) implies (B3) but it does not imply (B2). For $n=3$ and $d=1$, we show that (B1) does not imply (B3).
The reason for these negative results is in the asymptotic behaviour of the set $\{(x,r) \mid F(x)-rG(x) \succ 0\}$.

In Section \ref{sec6} we try to extend the Compact Positivstellensatz from one to several constraints.
The result is not satisfactory because it does not refer to the smallest weak preordering containing the constraints.
However, we obtain a satisfactiory version of Archimedean Positivstellensatz which may be of independent
interest because it generalizes the Scherer-Hol Theorem, see \cite{sh}, \cite{ks},\cite{c4}.

Finally, we remind the reader that there is another version of the Noncommutative Real Algebraic Geometry for matrix polynomials,
which is much more developed. 
The question there is the following: \textit{For given $G \in \sn$ characterize all $F \in \sn$ such that
$F(a)$ is positive definite for every $a \in \RR^d$ for which $G(a)$ is positive semidefinite.}
See \cite{sch},\cite{c6},\cite{viet} for the general case
and \cite{av},\cite{sh},\cite{ks},\cite{c4},\cite{zalar2} for the archimedean case.
Similar results also exist for some other algebras with involution;
see \cite{sch} for a survey.

\section{Geometric reformulation of the Questions}
\label{sec2}

Questions \ref{conj1} and \ref{conj2} can be geometricaly reformulated as follows:
\renewcommand{\theconj}{A'}

\begin{conj}
For which $F,G \in \sn$ are the following equivalent:
\begin{enumerate}
\item[(A1')] For every $a \in \RR^d$ there exists $r \in \RR$ such that $F(a)-r G(a) \succ 0$.
\item[(A2')] There exists a rational function $r(x)$ without singularities such that $F(a)-r(a)G(a) \succ 0$ for every $a \in \RR^d$.
\item[(A3')] There exists a rational function $r(x)$ without singularities in $L_G$ such that $F(a)-r(a)G(a) \succ 0$ for every $a \in L_G$.
\end{enumerate}
\end{conj}

\renewcommand{\theconj}{B'}

\begin{conj}
For which $F,G \in \sn$ are the following equivalent:
\begin{enumerate}
\item[(B1')] For every $a \in \RR^d$ there exists $r \in \RR$ such that $r>0$ and $F(a)-r G(a) \succ 0$.
\item[(B2')] There exists a rational function $r(x)$ without singularities such that $r(a)>0$ and $F(a)-r(a)G(a) \succ 0$ for every $a \in \RR^d$.
\item[(B3')] There exists a rational function $r(x)$ without singularities in $K_G$ such that $r(a)>0$ and $F(a)-r(a)G(a) \succ 0$ for every $a \in K_G$.
\end{enumerate}
\end{conj}

Lemma \ref{finsler1} implies that (A1) is equivalent to (A1') and Lemma \ref{finsler2} implies that
(B1) is equivalent to (B1'). We will also prove that (A2) is equivalent to (A2'), (B2) is equivalent to (B2'),
(A3) is equivalent to (A3') and (B3) is equivalent to (B3'). The proof of Lemma \ref{btoa} shows that
we can replace the condition $r>0$ in (B1)-(B3) with  $r \ge 0$.

We say that a subset $N$ of $\sn$ is a \textit{weak quadratic module} if $N+N \subseteq N$, $(\sos) \cdot N \subseteq N$ and $\sohs \subseteq N$.
For every $G \in \sn$, the set $N_G=\sum \mn^2+(\sos) \cdot G$ is the smallest weak quadratic module which contains $G$ and
$O_G=\sum \mn^2+\rx \cdot G$ is the smallest weak quadratic module which contains $G$ and $-G$.
In the case of $1 \times 1$ matrices, weak quadratic modules are exactly the usual quadratic modules. 

We will use the following trick several times:

\begin{lemma}
\label{htrick}
Suppose that  $N$ is a weak quadratic module in $\sn$ and $T$ is a preordering in $\rx$.
Then for every $F \in \sn$, the following are equivalent:
\begin{enumerate}
\item There exists $t \in T$ such that $tF \in I_n+T \cdot N$.
\item There exists $s \in \sos$ such that $(1+s)F \in I_n+T \cdot N$.
\end{enumerate}
\end{lemma}

\begin{proof}
By Lemma 3 in \cite{c2} there exists $h \in \sos$ such that $hI_n+F \in \sohs$.
Write $u:=1+(1+h)t \in 1+T$ and $s:=(1+h)u^2 \in \sos$. 
If we multiply $tF \in I_n+T \cdot N$ by $1+h$ and add $F$, 
we get that $uF \in I_n+T \cdot N$ which implies $u^2F \in I_n+T \cdot N$.
Once more, we multiply by $1+h$ and add $F$ to get $(1+s)F \in I_n+T \cdot N$.
This proves that (1) implies (2). The converse is clear.
\end{proof}

%The following Lemma can be deduced from Lemma \ref{htrick} but we give a different argument that we will use later.

\begin{lemma}
\label{btoa}
Take any $F,G \in \sn$ and write  $\tilde{G}=G \oplus -1$ and $\tilde{F}=F \oplus 0$.
Then for every preordering $T$ the following are equivalent:
\begin{enumerate}
\item There exists $s \in \sos$ such that $(1+s)F  \in I_n+T \cdot N_G$.
\item There exists $s \in \sos$ such that $(1+s)\tilde{F} \in I_n+T \cdot O_{\tilde{G}}$.
\end{enumerate}
\end{lemma}

\begin{proof}
Consider the following claim:
\begin{enumerate}
\item[(3)] $(1+s') F-(1+t') G \in I_n+T \cdot \sohs$ for some $s' \in \sos$ and $t' \in T$.
\end{enumerate}
Clearly, (3) implies (1). To prove the converse, pick $t \in T$ such that $(1+s)F -t G \in I_n+T \cdot \sohs$.
By Lemma 3 in \cite{c2}, there exists $h \in \sos$ such that $hI_n-G \in \sohs$. It follows that
$(1+h)(1+s)F-(1+(1+h)t)G \in I_n+T \cdot \sohs$.

On the other hand, (2) is equivalent to $(1+s) \tilde{F}-p \tilde{G} \in I_{n+1}+T \cdot \sohss$ for some $s \in \sos$ and $p \in \rx$.
The latter is equivalent to $(1+s) \tilde{F}-p \tilde{G} =I_{n+1}+W \oplus z$ for some $W \in T \cdot \sohs$ and $z \in T$, that is
$(1+s)F-pG=I_n+W$ and $p=1+z$ for some $W \in T \cdot \sohs$ and $z \in T$ which is exactly (3).
\end{proof}

\begin{comment}
\begin{proof}
Consider the following claims:
\begin{enumerate}
\item[(1')] $t F-t' G \in I_n+T \cdot \sohs$ for some $t,t' \in T$.
\item[(2')] $t \tilde{F}-p \tilde{G} \in I_{n+1}+T \cdot \sohss$ for some $t \in T$, $p \in \rx$.
\item[(3)] $t F-(1+t') G \in I_n+T \cdot \sohs$ for some $t,t' \in T$.
\end{enumerate}
By Lemma \ref{htrick}, (1) is equivalent to (1'), (2) is equivalent to (2') and (1') is equivalent to (3).
It remains to show that (3) is equivalent to (2'). Clearly, for every $t \in T$ and $p \in \rx$,
(2') is equivalent to $t \tilde{F}-p \tilde{G} =I_{n+1}+W \oplus z$ for some $W \in T \cdot \sohs$ and $z \in T$.
The latter is equivalent to $tF-pG=I_n+W$ and $p=1+z$ for some $W \in T \cdot \sohs$ and $z \in T$
which is exactly (3).
\end{proof}
\end{comment}

We will use several times the following version of Theorem 2 in \cite{c2}.
%The proof is by induction on $n$ using Schur complements. The converse is clear.

\begin{theorem}
\label{myLAApsatz0}
Suppose that a set $K \subseteq \RR^d$ and a preordering $T \subseteq \rx$ satisfy the following:
For every $f \in \rx$ such that $f(a) >0$ for all $a \in K$, there exists $t \in T$ such that $(1+t)f \in 1+T$.
Then for every $H \in \sn$ such that $H(a) \succ 0$ for all $a \in K$, there exists $t \in T$
such that $(1+t)H \in I_n+ T \cdot \sum \mn^2$. 
\end{theorem}

The converse is clear. In particular, we have the following:

\begin{cor}
\label{myLAApsatz1}
An element $H \in \sn$ satisfies $H(a) \succ 0$ for every $a \in \RR^d$
iff $(1+s)H \in I_n+\sum \mn^2$ for some $s \in \sos$.
\end{cor}

We are now able to explain the relations between the properties (A2), (A2'), (B2) and (B2').

\begin{prop}
\label{A2prop}
Any elements $F,G \in \sn$ satisfy property (A2) iff they satisfy property (A2'). Similarly,
they satisfy property (B2) iff they satisfy property (B2') iff $\tilde{F}:=F \oplus 0$ and $\tilde{G}:=G \oplus -1$
satisfy property (A2) iff $\tilde{F}$ and $\tilde{G}$ satisfy property (A2').
\end{prop}

\begin{proof}
Clearly, (A2) implies (A2'). Conversely, if $r=\frac{p}{q}$ satisfies (A2'), then $q^2>0$ and $q^2 F-pq G \succ 0$ everywhere.
By Corollary \ref{myLAApsatz1}, there exist $s_1,s_2,t \in \sos$ such that $(1+s_1)q^2=1+t$ and 
$(1+s_2)(q^2 F-pq G) \in I_n+\sohs$. It follows that $(1+t)(1+s_2)F \in I_n+O_G$.
%Finally, we use Lemma \ref{htrick} with $T=\sos$ and $N=O_G$ to get (A2).

By Lemma \ref{btoa}, $F$ and $G$ satisfy (B2) iff $\tilde{F}$ and $\tilde{G}$ satisfy (A2).
By Lemma \ref{finsler2}, $F$ and $G$ satisfy (B2') iff $\tilde{F}$ and $\tilde{G}$ satisfy (A2').
%We already know that $\tilde{F}$ and $\tilde{G}$ satisfy (A2) iff they satisfy (A2').
\end{proof}

\begin{comment}
\begin{prop}[Compact Positivstellensatz for nonbasic semialgebraic sets]
\label{cnonbasic}
Suppose that $S_1,\ldots,S_t$ are finite subsets of $\rx$ such that the set
$K:=K_{S_1} \cup \ldots \cup K_{S_t}$ is compact. Then the preordering
$T:=T_{S_1} \cap \ldots \cap T_{S_n}$ is archimedean and for every $H \in \sn$ the following are equivalent
\begin{enumerate}
\item $H(x)$ s positive definite for every $x \in K$.
\item There exists a real $\eps>0$ such that $H - \eps I_n \in T \cdot \sohs$.
\end{enumerate}
\end{prop}

\begin{proof}
Since $K$ is compact and $K_i$ are closed, $K_i$ are also compact. By Schm\" udgen's Positivstellensatz, 
it follows that each $T_{S_i}$ is archimedean. Therefore, $T$ is archimedean, and so, $T \cdot \sohs$
is archimedean as well. The equivalence of (1) and (2) now follows from the Hol-Scherer Theorem
(or my LAA paper).
\end{proof}
\end{comment}

To prove that (A3) is equivalent to (A3') we need an extension of Krivine-Stengle Positivstellensatz to nonbasic closed semiagebraic sets.
Recall that for every finite subset $S$ of $\rx$ we write $K_S:=\{a \in \RR^d \mid g(a) \ge 0$ for all $g \in S\}$
and $T_S$ for the preordering in $\rx$ generated by $S$. Later, we will also extend these definitions to subsets of $\sn$.
%Do not confuse this with the set $K_G$ defined above and the weak preordering $T_G$ to be defined later.

\begin{lemma}
\label{nonbasic}
Suppose that $S_1,\ldots,S_t$ are finite subsets of $\rx$. Then for every $h \in \rx$ we have that
$h(x)>0$ for every $x \in \bigcup_{i=1}^t K_{S_i}$ iff $(1+s)h \in 1+\bigcap_{i=1}^t T_{S_i}$ for some $s \in \sos$.
\end{lemma}

\begin{proof}
Clearly, the second assertion implies the first one. To prove the opposite, we will need the following claim:
\textit{For every finite subset $S$ of $\rx$ and for every polynomial $f \in \rx$ we have that
$f |_{K_S} > 0$ iff there exists $s \in \sos$ such that $(1+s)f \in 1+T_S$.} 

By the Krivine-Stengle Positivstellensatz, we have that $f |_{K_S} > 0$ iff there exists $t \in T_S$
such that $t f \in 1+T_S$. Now use Lemma \ref{htrick} with $T=N=T_S$ to get the claim.

Let us write $K:=  K_{S_1} \cup \ldots \cup K_{S_t}$. If $f \in \rx$ is such that $f |_K > 0$
then $f|_{K_{S_i}} >0$ for each $i$.  By the claim, there exist $s_1,\ldots,s_t \in \sos$ such that 
$(1+s_i) f \in 1+T_{S_i}$. Now, $s:=(1+s_1) \cdots (1+s_t)-1 \in \sos$ and $(1+s)f-1 \in T_{S_i}$ for each $i=1,\ldots,t$.
\end{proof}

As a corollary of Theorem \ref{myLAApsatz0} and Lemma \ref{nonbasic}, we get:

\begin{cor}
\label{myLAApsatz2}
Suppose that $S_1,\ldots,S_t$ are finite subsets of $\rx$. Then for every $H \in \sn$ we have that
$H(x) \succ 0$ for every $x \in \bigcup_{i=1}^t K_{S_i}$ iff
$(1+t)H \in I_n+(\bigcap_{i=1}^t T_{S_i}) \cdot \sum \mn^2$ for some $t \in \bigcap_{i=1}^t T_{S_i}$.
\end{cor}

In the proof of Proposition \ref{A3prop} we will also need the following.

\begin{lemma}
\label{kglg}
For every $G \in \sn$, the sets $K_G$ and $L_G$ are of the form $\bigcup_{i=1}^t K_{S_i}$ 
where $S_1,\ldots,S_t$ are finite subsets of $\rx$. 
Moreover, $L_G=K_G \cap K_{-G}$ and $K_G=L_{\tilde{G}}$
where $\tilde{G}:=G \oplus -1 \in \snn$.
\end{lemma}

\begin{proof}
Clearly, $L_G \subseteq K_G \cap K_{-G}$. To prove the opposite inclusion,
take any $a \in K_G \cap K_{-G}$ and pick nonzero $u,v \in \RR^n$ such that
$u^T G(a) u \ge 0$ and $v^T G(a)v \le 0$. It follows that the continuous function
$$\lambda \mapsto ((1-\lambda)u+\lambda v)^T G(a) ((1-\lambda)u+\lambda v)$$
has a zero $\lambda_0 \in [0,1]$. If $(1-\lambda_0)u+\lambda_0 v \ne 0$ then
we are done. Otherwise, $u$ and $v$ are colinear, which implies that 
$u^T G(a) u=v^T G(a) v=0$.  Therefore, $a \in L_G$ in this case, too.

The sets $K_G$ and $L_G$ are semialgebraic by the Tarski-Seidenberg Theorem.
The set $K_G$ consists of all $a \in \RR^d$ such that $G(a)$ is not negative definite.
Since $G$ is continuous and the set of negative definite matrices is open, it follows that
$K_G$ is closed. Since $L_G=K_G \cap K_{-G}$, $L_G$ is closed, too.
By the Finiteness Theorem, every closed semialgebraic set is a finite union 
of the sets of the form $K_S$ for finite $S$.

Finally, $a \in L_{\tilde{G}}$ iff $v^T \tilde{G}(a) v=0$ for some nonzero $v=(u,\alpha) \in \RR^{n+1}$
iff $u^T G(a) u-\alpha^2=0$ for some nonzeru $u \in \RR^n$ and some $\alpha \in \RR$ iff
$u^T G(a) u \ge 0$ for some nonzero $u \in \RR^n$ iff $a \in K_G$.
\end{proof}

\begin{prop}
\label{A3prop}
Any elements $F,G \in \sn$ satisfy property (A3) iff they satisfy property (A3').
Similarly, $F$ and $G$ satisfy property (B3) iff they satisfy property (B3') iff $\tilde{F}:=F \oplus 0$ and $\tilde{G}:=G \oplus -1$
satisfy property (A3) iff $\tilde{F}$ and $\tilde{G}$ satisfy property (A3').
\end{prop}

\begin{proof}
If $F$ and $G$ satisfy (A3), then $(1+s)F=I_n+\sum_i u_i (V_i+t_iG)$ for some $s \in \sos$, $V_i \in \sohs$, $u_i \in \pos(L_G)$ and $t_i \in \rx$.
Then (A3') is satisfied with $r=\frac{\sum_i u_i t_i}{1+s}$. 

Suppose now that $F$ and $G$ satisfy (A3'). Then there exist $p,q \in \rx$ such that
$q^2>0$ on $L_G$ and $H:=q^2F-pqG \succ 0$ on $L_G$. By Lemma \ref{kglg}, $L_G=\bigcup_{i=1}^t K_{S_i}$
for finite subsets $S_1,\ldots,S_t \subset \rx$. By Lemma \ref{nonbasic},
there exist $s \in \sos$ and $t_1 \in T:=\bigcap_{i=1}^t T_{S_i}$ such that $(1+s)q^2=1+t_1$.
By Corollary \ref{myLAApsatz2}, there exists $t_2 \in T$ such that $(1+t_2)H \in I_n+T \cdot \sohs$.
It follows that $(1+t_1)(1+t_2)F \in I_n+T \cdot \sohs+\rx \cdot G \subseteq I_n+T \cdot O_G$.
If we use Lemma \ref{htrick} with $N=O_G$ and the fact that $T \subseteq \pos(L_G)$ we get (A3).

Lemma \ref{btoa} with $T=\pos(K_G)=\pos(L_{\tilde{G}})$ implies that 
$F$ and $G$ satisfy property (B3) iff $\tilde{F}$ and $\tilde{G}$ satisfy property (A3).
By Lemma \ref{finsler2}, $F$ and $G$ satisfy property (B3') iff $\tilde{F}$ and $\tilde{G}$ satisfy property (A3').
\end{proof}

\section{Asymptotic versions of the Questions}
\label{sec3}

We would like to determine when the property (A1') implies (A2'). For given $F,G \in \sn$ consider the set
$$M:=\{(x,r) \in \RR^d \times \RR \mid F(x)-r G(x) \succ 0\}.$$
The assumption (A1') says that all sections $M_x:=\{r \in \RR \mid (x,r) \in M\}$ are nonempty.
The conclusion (A2') says that there exists a rational function $r$ without singularities such that
$r(x) \in M_x$ for every $x \in \RR^d$ (i.e. $M$ contains the graph of $r$.)

In the following we will write $B(a,\delta)$ (resp. $\bar{B}(a,\delta)$) for the open ball
(resp. closed ball) with center $a$ and radius $\delta$ in the euclidian norm $\Vert \cdot \Vert$.
We will also apply the euclidean norm to the $d$-tuple of variables.

\begin{lemma}
\label{mainlem}
Suppose that $F,G \in \sn$ satisfy (A1') and write $M_x:=\{r \in \RR \mid F(x)-rG(x) \succ 0\}$ 
and $\mu(x)=\inf M_x,\nu(x)=\sup M_x$ for every $x \in \RR^d$. We claim that
\begin{enumerate}
\item $M_x=(\mu(x),\nu(x))$ for every $x \in \RR^d$.
\item $\mu(x) = -\infty$ iff $G(x) \succeq 0$ and $\nu(x)=+\infty$ iff $G(x) \preceq 0$.
\item $\mu,\nu \colon \RR^d \to \overline{\RR}$ are continuous  functions.
\item For every compact set $K \subseteq \RR^d$ there exists a polynomial $p$ such that $p(x) \in M_x$ for every $x \in K$.
\end{enumerate}
\end{lemma}

\begin{proof}
Since the set of positive definite matrices is convex and open, so are the sets $M_x$. This implies (1).

If $G(x) \succeq 0$ and $r_0 \in M_x$ for some $x$, then for every $r \le r_0$ we have that $r \in M_x$ since
$F(x)-r G(x)=F(x)-r_0 G(x)+(r_0-r)G(x) \succ 0$. Conversely, if $\mu(x) = -\infty$, then by the convexity
of $M_x$, we have that $(-\infty,r_0] \subseteq M_x$ for some $r_0$. It follows that
$\frac{1}{r_0-r} (F(x)-r_0 G(x))+G(x) \succ 0$ for every $r<r_0$. Sending $r \to -\infty$, we get
that $G(x) \succeq 0$.

Let us show that $\mu$ is continuous. Suppose first that $\mu(a)=-\infty$ and pick $r_0 \in M_a$.
Since the set $M:=\{(x,r) \in \RR^d \times \RR \mid F(x)-r G(x) \succ 0\}$ is open, we have that
for every $r \le r_0$ there exists $\delta>0$ such that $B(a,\delta) \times r \subset M$.
It follows that $\mu(x) < r$ for every $x \in B(a,\delta)$. Suppose now that $\mu(a) \ne -\infty$ 
and pick $\eps>0$. Take any $\kappa \in (\mu(a)-\eps,\mu(a))$ and any $\lambda \in (\mu(a),\mu(a)+\eps) \cap M_a$.
Since $M$ is open, there exists $\delta_1>0$ such that $B(a,\delta_1) \times \lambda \subset M$.
Note also that $u^T(F(a)-\kappa G(a))u<0$ for some nonzero $u \in \RR^n$. Otherwise, we would have that
$F(a)-\kappa G(a) \succeq 0$, which would imply a contradiction $F(a)-\mu(a) G(a) \succ 0$
since $F(a)-\lambda G(a) \succ 0$ and $\mu(a) \in (\kappa,\lambda)$. Now pick $\delta_2>0$
such that $u^T(F(x)-\kappa G(x))u<0$ for every $x \in B(a,\delta_2)$. It follows that 
$\kappa < \mu(x) < \lambda$ for every $x \in B(a,\min\{\delta_1,\delta_2\})$. 

To prove (4) first use a compactness argument to construct $c,d \in \RR$ such that $M_x \cap (c,d)$
is nonempty for every $x \in K$. It follows that $\max\{\mu,c\}$ and $\min\{\nu,d\}$ are continuous
and finite on $K$. By the Stone-Weierstrass theorem, there is a polynomial on $K$ between them.
\end{proof}

Our main results will follow from Proposition \ref{main}.

\begin{prop}
\label{main}
Suppose that $G \in \sn$ is such that $G(a)$ is negative semidefinite for every $a$ outside some ball in $\RR^d$.
Then for every $F \in \sn$ property (A1) implies (A2) and (B1) implies (B2).
\end{prop}

\begin{proof} We will prove that the property (A1') implies the property (A2'). 
It follows that (B1') implies (B2') because $\tilde{G}:=G \oplus -1$ is also
negative semidefinite outside the same ball as $G$.

Suppose that $G(x) \preceq 0$ for every $x$ outside the ball $B(0,R)$. By Lemma \ref{mainlem},
there exists a polynomial $p$ such that $F(x)-p(x) G(x) \succ 0$ for every $x \in B(0,R+1)$.
The function $\max\{\mu,p\}$ is defined and finite everywhere. It is also continuous by Lemma \ref{mainlem} 
and semialgebraic by the Tarski-Seidenberg Theorem. By Proposition 2.6.2. in [BCR], 
it is bounded from above by a polynomial of the form $q(x)=C(1+\Vert x \Vert^2)^t$.

By a compactness argument there exists $\eps>0$ such that 
$F(x)-(p(x)+\eps) G(x) \succ 0$ for every $x \in B(0,R)$. For every $k \in \NN$ write
$$p_k(x):=p(x)+\eps\cdot\left(\frac{1+\Vert x \Vert^2}{1+R^2} \right)^k.$$
Since $p_k \le p+\eps$ on $B(0,R)$ and $\sup M_x=+\infty$ on $B(0,R+1)\setminus B(0,R)$, we have that 
$F-p_k G \succ 0$ on $B(0,R+1)$ for every $k$. 
Finally, pick $k \in \NN$ such that $p_k(x) > q(x)$ for every $x$ outside $B(0,R+1)$.
For this $k$ we have that $F-p_k G \succ 0$  on $\RR^d$. Therefore, (A2') is true.
\end{proof}

The assumption that $G$ is negative semidefinite outside some ball cannot be omitted as the following example shows.

\begin{ex}
\label{exmain}
If 
$$G=\left[ \begin{array}{cc} x & 0 \\ 0 & x \end{array} \right]
\quad \text{ and } \quad
F=\left[ \begin{array}{cc} 1+x & 0 \\ 0 & 1 \end{array} \right]$$
and $\tilde{F}=F \oplus 0$ and $\tilde{G}=G \oplus -1$ then
all sections of the sets $M$ and $\tilde{M}:=\{(x,r) \in \RR^d \times \RR \mid \tilde{F}(x)-r \tilde{G}(x) \succ 0\}$
are nonempty since
$$M_x=\left\{\begin{array}{cc} (1+\frac{1}{x},+\infty), & x<0, \\
(-\infty,+\infty), & x=0, \\ (-\infty, \frac{1}{x}), & x>0, \end{array} \right. \quad
\tilde{M}_x=\left\{\begin{array}{cc} (1+\frac{1}{x},+\infty), & x<-1, \\
(0,+\infty), & -1 \le x \le 0, \\ (0, \frac{1}{x}), & x>0. \end{array} \right.$$
\begin{figure}
\centering
\begin{minipage}[c]{4cm}
\includegraphics[width=3.5cm, natwidth=12cm,natheight=12cm]{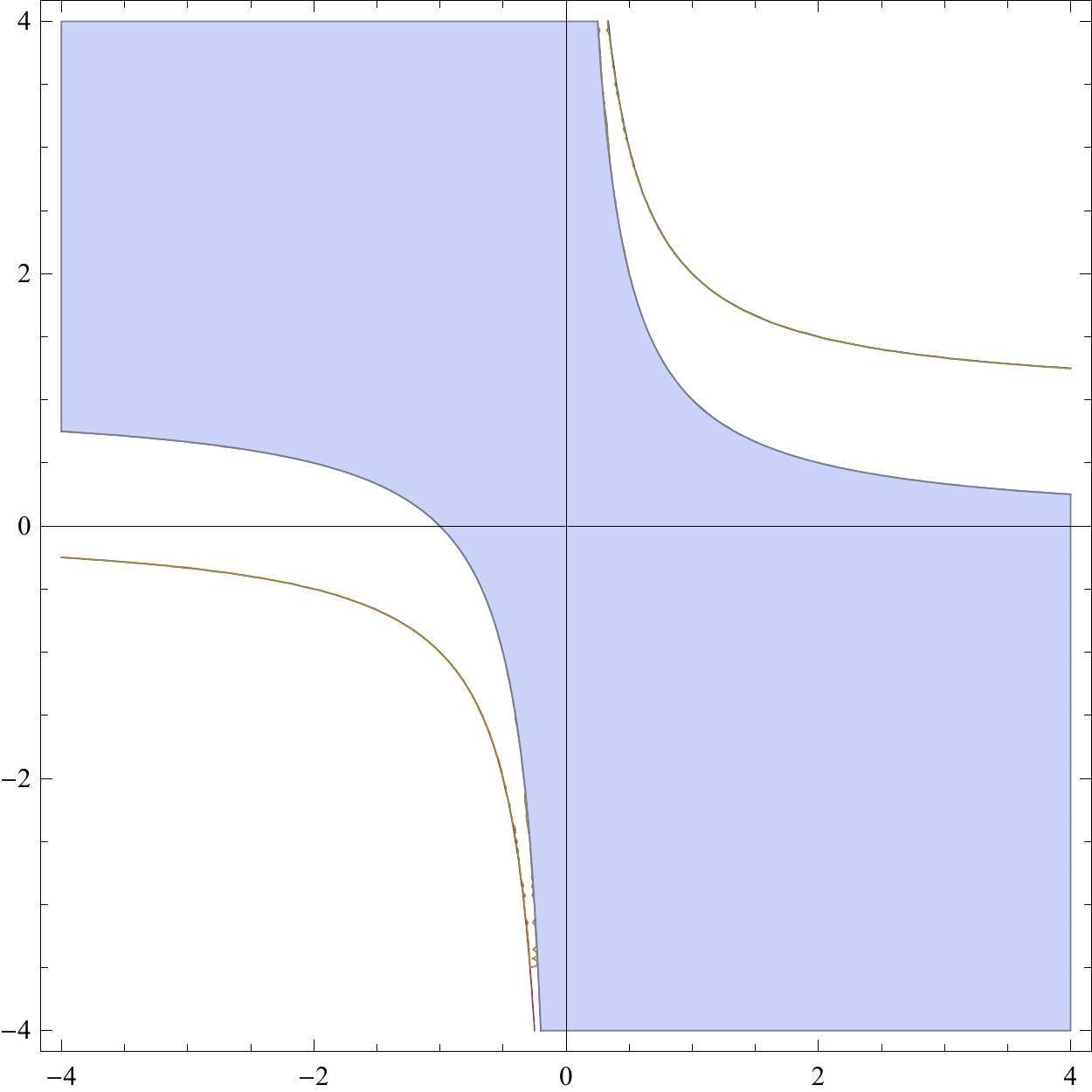}
\centerline{\small FIG.\,1\enspace Set $M$ from Ex.\,\ref{exmain}}
\end{minipage}
\hskip 3cm
\begin{minipage}[c]{4cm}
\includegraphics[width=3.5cm, natwidth=12cm,natheight=12cm]{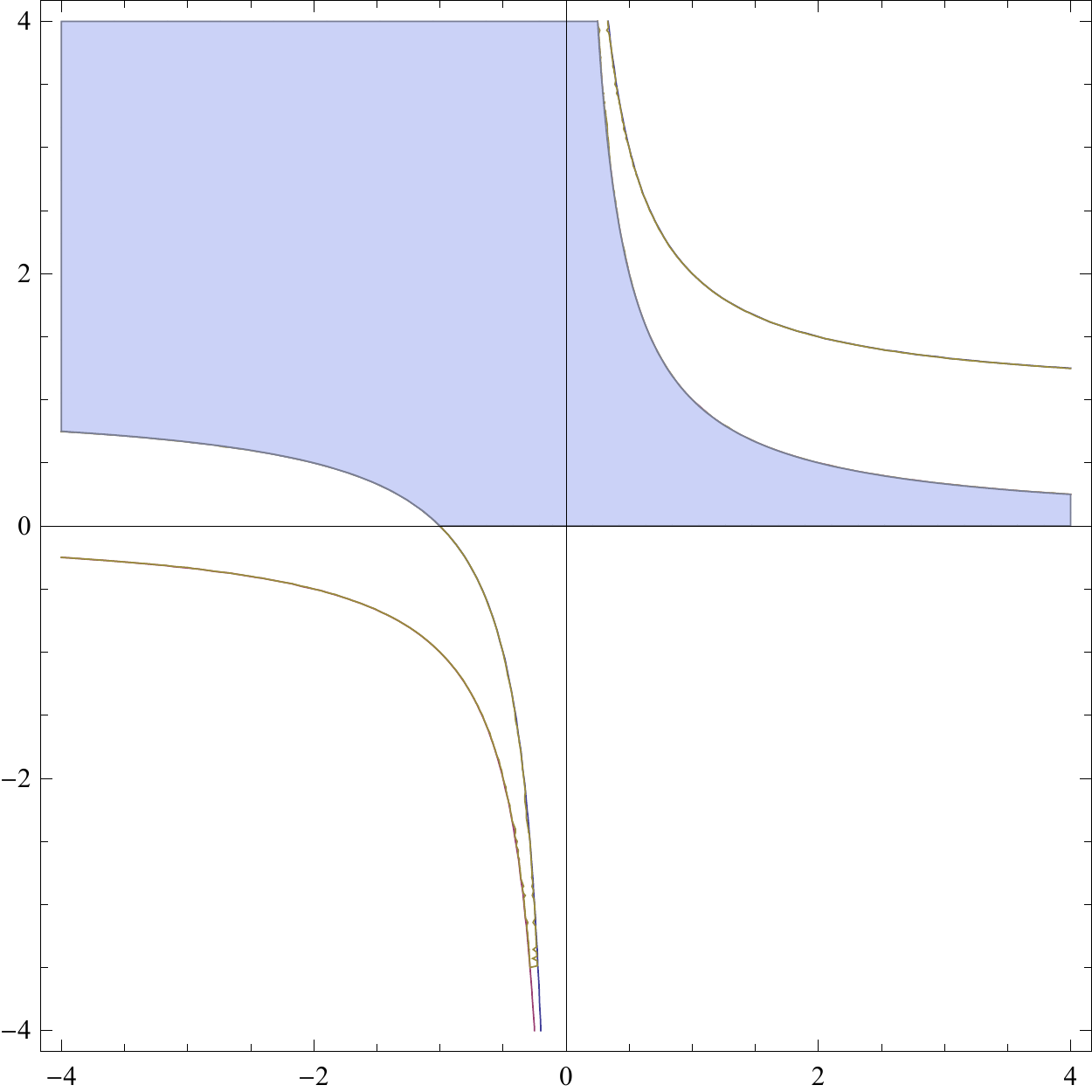}
\centerline{\small FIG.\,2\enspace Set $\tilde{M}$ from Ex.\,\ref{exmain}}
\end{minipage}
\end{figure}
On the other hand, $-x \in M_x$ for every $x \in \RR^d$ but there is no
rational function without singularities such that $r(x) \in \tilde{M}_x$ for every $x \in \RR^d$.
Namely, the existence of such $r(x)$ would imply a contradiction
$$\lim\limits_{x \to +\infty} r(x)=0 \quad  \text{ and } \quad \lim\limits_{x \to -\infty} r(x) \ne 0.$$
It follows that $F$ and $G$ satisfy (A1'), (A2') and (B1') but not (B2').
\end{ex}

The reason for the failure of Property (B2') in the example is in the asymptotic behavior of the set $M$.
The following generalization of Proposition \ref{main} shows us that the asymptotic behaviour of the set $M$
is the only obstruction.

\begin{prop}
\label{containsrat2}
Matrix polynomials $F,G \in \sn$ satisfy property (A2') iff they satisfy property (A1') and the following:
\begin{enumerate}
\item[(A2'')] There exists a real $R$ and a rational function $r$ such that $r$ has no singularities outside $B(0,R)$
and $F-rG \succ 0$ outside $B(0,R)$.
\end{enumerate}
Moreover, $F$ and $G$ satisfy  (A3') iff they satisfy  (A1') and the following:
\begin{enumerate}
\item[(A3'')] There exists a real $R$ and a rational function $r$ such that $r$ has no singularities in $K_G \setminus B(0,R)$
and $F-rG \succ 0$ on $K_G \setminus B(0,R)$.
\end{enumerate}
We also have analogous results for Question \ref{conj2}.
\end{prop}

\begin{proof}
We will only show the second part because the first part is similar (just replace $K_G$ with $\RR^d$.) The results for Question \ref{conj2} follow.
If the set $K_G$ is compact we can  use Proposition \ref{main}. Otherwise the set $Z:=K_G\setminus B(0,R)$ is nonempty. It is also closed and semialgebraic.

Suppose that $F,G$ satisfy (A3''). As in the proof of Proposition \ref{A3prop}, we see that there is $s \in \sos$ such that
$(1+s)F \in I_n+\pos(Z) \cdot N_G.$ Thus, there exists a rational function $r_0$ with no singularities on $\RR^d$ such that $F-r_0 G \succ 0$ on $Z$. 
%In the notation of Lemma \ref{mainlem} this means that $r_0(x) \in M_x=(\mu(x),\nu(x))$ for every $x \in Z$.

By a compactness argument, we can find a real $c>0$ such that the interval $(r_0(x)-c,r_0(x)+c)$ 
intersects the interval $M_x:=\{r \in \RR \mid F(x)-rG(x) \succ 0\}=(\mu(x),\nu(x))$ for every $x \in B(0,R)$.
It follows that the functions $\mu_1:=\max\{\mu,r_0-c\}$ and $\nu_1:=\min\{\nu,r_0+c\}$ satisfy $\mu_1<\nu_1$ on $B(0,R)$.
Clearly, $\mu_1<r_0<\nu_1$ on $Z$. Moreover, $\mu_1$ and $\nu_1$ are defined, finite, continuous and semialgebraic on $\RR^d$.

By Proposition 2.6.2. in \cite{bcr}, there exists a real $d>0$ and an integer $k\ge 1$ such that 
$\frac{1}{\nu_1-r_0}$  and $\frac{1}{r_0-\mu_1}$ are below $d (1+\Vert x \Vert^2)^k$ on $Z$.
So, $r_0(x) \pm \frac{1}{d} (1+\Vert x \Vert^2)^{-k} \in (\mu_1(x),\nu_1(x))\subseteq M_x$ for every $x \in Z$.
Write $\sigma(x):=c\big(\frac{1+R^2}{1+\Vert x \Vert^2} \big)^{k+1}$  and pick $R'>0$ such that
$\sigma(x) \le \frac{1}{d} (1+\Vert x \Vert^2)^{-k}$ for every $x$ outside $B(0,R')$.

Let us write $\mu_2:=\max\{\mu_1,r_0-\sigma\}$, $\nu_2:=\min\{\nu_1,r_0+\sigma\}$ and $\phi:=\frac{\mu_2+\nu_2}{2}$.
For every $x \in B(0,R)$, we have that $\sigma(x) \ge c$, so $\mu_2(x)=\mu_1(x)$ and $\nu_2(x)=\nu_1(x)$.
Clearly, $r_0(x) \in (\mu_2(x),\nu_2(x)) \subseteq M_x$ for every $x \in Z$.
It follows that $\mu_2<\nu_2$ on $K_G$. 
For every $x \not\in B(0,R')$, we have that $\mu_2(x)=r_0(x)-\sigma(x)$ and $\nu_2(x)=r_0(x)+\sigma(x)$
and $\phi(x)=r_0(x)$.

Let $C_0(K_G)$ be the algebra of all real continuous functions on $K_G$ that vanish at infinity
and let $\mathcal{A} \subseteq C_0(K_G)$ be the algebra of (the restrictions to $K_G$ of) all rational functions
of the form $\frac{h(x)}{(1+\Vert x \Vert^2)^l}$ where $\deg h< 2l$. Since $\mathcal A$
separates points and vanishes nowhere, it is dense in $C_0(K_G)$ in the sup norm by the
locally compact version of the Stone-Weierstrass Theorem. 

Pick a real $\eps \in (0,1)$ such that $2 \eps c(1+R^2)^{k+1} < \nu_2-\mu_2$ on $B(0,R')$.
It follows that $2 \eps \sigma<\nu_2-\mu_2$ on $B(0,R')$ and $2 \eps \sigma<2 \sigma=\nu_2-\mu_2$ on $K_G \setminus B(0,R')$,
thus $(\phi(x) -\eps \sigma(x),\phi(x) + \eps \sigma(x))\subseteq (\mu_2(x),\nu_2(x))  \subseteq M_x$ for every $x \in K_G$.
The function $f(x):=(1+\Vert x \Vert^2)^{k+1} (\phi-r_0)$ belongs to $C_0(K_G)$ because it is identically zero on $K_G \setminus B(0,R')$.
By the above, there exists $g \in \mathcal{A}$ such that $g(x)$ lies between 
$f(x) \pm \eps c(1+R^2)^{k+1}$ on $K_G$. 
It follows that for every $x \in K_G$, $r_1(x):=r_0(x)+\frac{g(x)}{(1+\Vert x \Vert^2)^{k+1}}$
belongs to $(\phi(x) -\eps \sigma(x),\phi(x) + \eps \sigma(x))\subseteq M_x$.
\end{proof}

\section{Applications}
\label{sec4}

In this section we will discuss two applications of Proposition \ref{main}.
%, Theorems \ref{mainth1} and \ref{mainth2}.
Theorem \ref{mainth1} extends the usual characterization of positive polynomials on varieties.
It is related to the one-sided Real Nullstellensatz for matrix polynomials from \cite{c3}.

\begin{theorem}[Positivstellensatz for varieties]
\label{mainth1}
Pick $G_1,\ldots,G_m \in \mn$ and let $J$ be the left ideal generated by them.
Then for every $F \in \sn$ the following are equivalent:
\begin{enumerate}
\item[(1)] For every $a \in \RR^d$ and every nonzero $v \in \RR^n$ such that $G_1(a)v=\ldots=G_m(a)v=0$ we have that $v^T F(a)v>0$.
\item[(2)] There exists $s \in \sos$ such that $(1+s)F \in I_n+\sum \mn^2+J+J^T$.
\end{enumerate}
\end{theorem}

\begin{proof} 
Clearly, (2) implies (1). To prove the opposite, write $G=-\sum_{i=1}^m G_i^T G_i$ and consider the following claims:
\begin{enumerate}
\item[(1')] For every $a \in \RR^d$ and every nonzero $v \in \RR^n$ such that $v^T G(a) v = 0$ we have that $v^T F(a)v>0$.
\item[(2')] There exist  $s\in \sos$, $S \in \sohs$ and $p \in \rx$ such that $(1+s)F=I+S+pG$.
\end{enumerate}
 By Proposition \ref{main}, (1') implies (2'). Clearly, (1) is equivalent to (1') and (2') implies (2).
\end{proof}

Theorem \ref{mainth2} extends the usual characterization of positive polynomials on compact sets, i.e. Schm\" udgen's Theorem.

\begin{theorem}[Compact Positivstellensatz]
\label{mainth2}
Suppose that $G \in \sn$ is such that the set $K_G$ is compact.
Then for every $F \in \sn$ the following are equivalent:
\begin{enumerate}
\item[(1)] For every $a \in \RR^d$ and every nonzero $v \in \RR^n$ such that $v^T G(a) v \ge 0$ we have that $v^T F(a)v>0$.
\item[(2)] There exists $\eps>0$ such that $F -\eps I_n \in N_G$.
\end{enumerate}
\end{theorem}

% $N_G = \sum \mn^2+(\sos) \cdot G$
% $K_G:=\{a \in \RR^d \mid v^T G(a)v \ge 0$ for some nonzero $v \in \RR^n\}$

As usual, we will split the proof into two parts, W\" ormann's trick and Archimedean Positivstellensatz,
see Propositions \ref{wtrick} and \ref{flct}.

We say that a weak quadratic module $N \subseteq \sn$ is \textit{archimedean} if for every 
$F \in \sn$ there exists $k \in \NN$ such that $kI_n \pm F \in N$. As usual, one can show
that $N$ is archimedean iff  $(R^2-\Vert x \Vert^2)I_n \in N$ for some real $R$.
(See the proof of Theorem 2.1 in \cite{c1} for details.)

Proposition \ref{flct} is a special case of Theorem 2.1. in \cite{c1}, 
Finsler's Lemma implies that property (1) from Proposition \ref{flct} is equivalent to property (1) from Theorem \ref{mainth2}.

\begin{prop}
\label{flct}
Suppose that $G \in \sn$ is such that the weak quadratic module $N_G$ is archimedean.
Then for every $F \in \sn$ the following are equivalent:
\begin{enumerate}
%\item[(1)] For every $a \in \RR^d$ and every nonzero $v \in \RR^n$ such that $v^T G(a) v \ge 0$ we have that $v^T F(a)v>0$.
\item[(1)] For every $a \in \RR^d$ and a nonzero positive semidefinite real $n \times n$ matrix $B$ such that $\tr G(a)B \ge 0$, we have that $\tr F(a)B >0$.
\item[(2)] There exists $\eps>0$ such that $F -\eps I_n \in N_G$.
\end{enumerate}
\end{prop}

Proposition \ref{wtrick} follows from Proposition \ref{main}.

\begin{prop}
\label{wtrick}
For every $G \in \sn$, the weak quadratic module $N_G$ is archimedean iff the set $K_G$ is compact.
\end{prop}

\begin{proof}
If $N_G$ contains $(R^2-\Vert x\Vert^2)I_n$ for some $R$ then $K_G$ is contained in $\bar{B}(0,R)$. 
Conversely, If $K_G$ is contained in $B(0,R)$ for some $R$, then $f:= R^2-\Vert x\Vert^2 >0$ on $K_G$.
It follows from Proposition \ref{main} that there exists  $s \in \sos$ such that $(1+s) f \cdot I_n \in I_n+N_G$.
Write $N':=\{p \in \rx \mid p I_n \in N_G\}$ and $T_f=\sos+f (\sos)$. Clearly, $N'$ and $T_f$ are
quadratic modules in $\rx$. Since $T_f$ is archimedean, there exists $M>0$ such that $M-s \in T_f$.
Since $(1+s)f \in 1+N'$, it follows that $(1+s)(M-s) \in N'$, and so 
$R^2 (\frac{M}{2} +1)^2-\Vert x\Vert^2=R^2(M-s)(1+s)+R^2(\frac{M}{2}-s)^2+(1+s)(R^2-\Vert x\Vert^2)+s \Vert x\Vert^2 \in N'$.
This proves that $N'$ is archimedean. Therefore, $N_G$ is archimedean, too.
\end{proof}

To make the reader appreciate  this  argument, we also give a completely algebraic proof
of Proposition \ref{wtrick}. It uses the following:

\begin{lemma}
\label{oneplusq}
For every $G \in \sn$ and $f \in \rx$ such that $f>0$ on $K_G$
there exists $Q \in \sohs$ such that $(I_n+Q)f \in I_n+ N_G$.
\end{lemma}

\begin{proof}
The assumption implies that $-G(a) \succ 0$ for every $a \in \RR^d$ such that $-f(a) \ge 0$.
By Corollary \ref{myLAApsatz2},  there exists $t \in T_f$ and $S_0,S_1 \in \sohs$ such that 
$(1+t)(-G)=I_n +S_0+(-f) S_1$. It follows that $(1+t)^2(-G)=I_n +S_0'+(-f) S_1'$ for some
$S_0',S_1' \in \sohs$.
Write $h_\pm=(1 \pm f)^2/4$.
Multiplying $f S_1'=I_n+S_0'+(1+t)^2 G$ with $1+h_-^2=(1+h_+^2)-f$ one gets 
$(1+h_-^2) f S_1'=(1+h_+^2-f)I_n+(1+h_-^2)S_0'+(1+h_-^2)(1+t)^2 G$. Thus,
$(1+(1+h_-^2)S_1')f \in I_n+N_G$. 
\end{proof}

We are now able to give an alternative proof of Proposition \ref{wtrick}. 

\begin{proof}
Since $K_G$ is compact, it is contained in some ball $B(0,R)$. Since $f:=R^2-\Vert x \Vert^2$ is $>0$ on $K_G$,
there exists $Q \in \sohs$ such that 
\begin{equation}
\label{eq1}
f(I_n+Q) \in I_n+N_G.
\end{equation}
By Lemma 3 in \cite{c2}, there exists $q \in \sos$ such that $q I_n-Q \in \sohs$. Since the preordering $T'=\sos+f(\sos)$ is archimedean, there exist
$m \in \NN$ and $s_0,s_1 \in \sos$ such that $m-1-q=s_0 +s_1 f$. From \eqref{eq1}, it follows that
\begin{equation}
\label{eq2}
(m-1-q)(I_n+Q) \in N_G.
\end{equation}
Since $(1+q)I_n-Q$ and $I_n+Q$ commute and they are positive definite at each point, it follows that $((1+q)I_n-Q)(I_n+Q)$ is positive
definite at each point. By Corollary \ref{myLAApsatz1}, there exists $s \in \sos$ such that
\begin{equation}
\label{eq3}
(1+s)((1+q)I_n-Q)(I_n+Q) \in I_n+\sohs.
\end{equation}
From the equations \eqref{eq2} and \eqref{eq3}, we get that
\begin{equation}
\label{eq4}
(1+s)(mI_n-Q)(I_n+Q) \in N_G.
\end{equation}
By adding $(1+s)(\frac{m}{2}I_n-Q)^2 \in \sohs$ to \eqref{eq4}, we get that
\begin{equation}
\label{eq5}
(1+s)((\frac{m^2}{4}+m)I_n-Q) \in N_G.
\end{equation}
From \eqref{eq5} and \eqref{eq1}, we get that
\begin{equation}
\label{eq6}
R^2 (1+s)((\frac{m^2}{4}+m)I_n-Q)+(1+s) f(I_n+Q)+(1+s) \Vert x \Vert^2 Q \in N_G
\end{equation}
which can be simplified to
\begin{equation}
\label{eq7}
(1+s)(R^2(\frac{m}{2}+1)^2-\Vert x \Vert^2 )I_n \in N_G.
\end{equation}
As in the first proof of Proposition \ref{wtrick}, we can deduce from \eqref{eq7} that $N_G$ is archimedean.
\end{proof}

\section{Weak preorderings}
\label{sec5}

The aim of this section is to provide motivation for the study of property (B3) and for Section \ref{sec6}.

A weak quadratic module $T$ is a \textit{weak preordering} if $(T \cap Z) \cdot T \subseteq T$ where the set
$Z:=\rx \cdot I_n$ will be identified with the set $\rx$ in the sequel. In the case of $1 \times 1$ matrices,
weak preorderings are exactly the usual preorderings. 
The smallest weak preordering which contains $G$ will be denoted by $T_G$. Proposition \ref{tgprop} gives a recursive construction of $T_G$.

\begin{prop}
\label{tgprop}
Pick $G \in \sn$ and consider the sequence
$$T_0= \sos,\quad T_{i+1}=\prod ((T_i \cdot N_G) \cap Z)$$
of subsets in $\rx$. In other words, $T_{i+1}$ is the set of all finite products
of elements $t \in \rx$ such that $t I_n \in T_i \cdot N_G$. 
We claim that 
\begin{enumerate}
\item $T:= \bigcup_i T_i$ is a preordering in $\rx$.
\item $T_G=T \cdot N_G$.
\item $T_G \cap Z=T.$
\item $T \subseteq \pos(K_G)$ where $K_G:=\{a \in \RR^d \mid v^T G(a)v \ge 0$ for some nonzero $v \in \RR^n\}$.
\item $\pos(K_G) \cdot N_G$ is a preordering which contains $T_G$.
\item If $G=g I_n$ then $T_G=T_g \cdot \sohs=\sohs+g \sohs$ where $T_g=\sos+(\sos) g$.
\end{enumerate}
\end{prop}

\begin{proof}
(1) The set $T_0$ is clearly a preordering in $\rx$. If $T_i$ is a preordering in $\rx$
then the set $(T_i \cdot N_G) \cap Z$ is clearly a quadratic module in $\rx$ containing $T_i$.
It follows that $T_{i+1}$ is a preordering in $\rx$ containing $T_i$. Therefore $T$ is a preordering in $\rx$.

(2) Clearly, $T_0 \cdot N_G \subseteq T_G$. Suppose that $T_i \cdot N_G \subseteq T_G$. It follows that
$T_{i+1}=\prod ((T_i \cdot N_G) \cap Z) \subseteq \prod (T_G \cap Z)=T_G \cap Z$. Therefore,
$T_{i+1} \cdot N_G \subseteq (T_G \cap Z) \cdot T_G \subseteq T_G$. This proves the inclusion $T \cdot N_G \subseteq T_G$.
To prove the opposite inclusion, it suffices to show that $T \cdot N_G$ is a preordering. Pick $u \in (T_i \cdot N_G) \cap Z$
and $V \in T_j \cdot N_G$ and write $k=\max\{i,j\}$. It follows that $u \in T_{k+1}$, and so, $uV \in T_{k+1} \cdot (T_k \cdot N_G)
=T_{k+1} \cdot N_G \subseteq T \cdot N_G$. 

(3) follows from $T \subseteq T_G \cap Z = \bigcup_i ((T_i \cdot N_G) \cap Z)  \subseteq \bigcup T_{i+1}=T$.

(4) Clearly, $T_0 \subseteq \pos(K_G)$. Suppose that $T_i \subseteq \pos(K_G)$ and pick any $z \in (T_i \cdot N_G) \cap Z$
and any $a \in K_G$. Pick $s_j \in T_i$ and $H_j \in N_G$ such that $zI_n=\sum_j s_j H_j$. Pick also a nonzero $v \in \RR^d$ 
such that $v^T G(a) v \ge 0$. It follows that $z(a)v^Tv= \sum_j s_j(a) v^T H_j(a)v$. Clearly $v^T H_j(a) v \ge 0$ and
by induction hypothesis, also $s_j(a) \ge 0$. Therefore, $z(a) \ge 0$. It follows that $T_{i+1} \subseteq \pos(K_G)$.

(5) It suffices to show that $(\pos(K_G) \cdot N_G) \cap Z \subseteq \pos(K_G)$. Pick any $h \in \pos(K_G) \cdot N_G$ any $a \in K_G$.
Then there exists nonzero $v \in \RR^d$ such that $v^T G(a) v \ge 0$. Since $h \in \pos(K_G) \cdot N_G$, we have that $h=\sum_j p_j H_j$,
where $p_j \in \pos(K_G)$ and $H_j \in N_G$. It follows that $h(a) v^T v=\sum_j p_j(a) v^T H_j(a) v \ge 0$. It follows that $h(a) \ge 0$.

(6) follows from  $(T_g \cdot \sohs) \cap Z \subseteq T_g$ which is clear.
\end{proof}

%Note that $\pos(K_G) \cdot N_G$ may not be proper when $T_G$ is proper.

Suppose that $F,G \in \sn$ and consider the following property:

\begin{enumerate}
\item[(B4)] There exists $s \in \sos$ such that $(1+s)F \in I_n+T_G$.
\end{enumerate}
Since $N_G \subseteq T_G \subseteq \pos(K_G) \cdot N_G$ by Proposition \ref{tgprop}, property (B2) implies property (B4) and property (B4) implies property (B3). 
We will show that 
\begin{itemize}
\item (B1) does not imply (B4) for $n \ge 2$. It suffices to show this for $n=2$; see Example \ref{exa}.
\item (B1) does not imply (B3) for $n  \ge 3$. It suffices to show this for $n=2$; see Example \ref{exb}.
\item (B1) implies (B3) if $n=2$ and $d=1$; see Proposition \ref{containsrat3}.
\end{itemize}

%The following example show that (1) does not always imply (3) for $n=2$.

\begin{ex}
\label{exa}
Clearly, $G=x \cdot \left[\begin{array}{cc} 2 & 0 \\ 0 & 1 \end{array} \right]$ and $F=(1+x) \cdot I_2$ satisfy property (B1).
We will show that they do not satisfy property (B4). 
%(It follows that (B1) is satisfied but (B2) is not.)
Since $1+x$ has odd degree, it suffices to show that every element of $T_G \cap Z$
has even degree and positive leading coefficient. By Proposition \ref{tgprop}, it suffices to show that every element of every preordering
$T_i$ has even degree and positive leading coefficient. This is clearly true for $T_0$. Suppose it is true for $T_i$ and take any element
$z \in (T_i \cdot N_G) \cap Z$. Then  there exist elements $s_i \in T_i$ and $\alpha_i,\beta_i,\gamma_i \in \sos$ such that
$z \oplus z=\sum_i s_i((\alpha_i \oplus \beta_i) + \gamma_i G)$. If we write $t_1=\sum_i s_i \alpha_i$, $t_2=\sum_i s_i \beta_i$
and $s=\sum_i s_i \gamma_i$, then $t_1,t_2,s \in T_i$ and $z=t_1+2 x s=t_2+x s$.
By the induction hypothesis, $t_1,t_2$ and $s$ have even degree and positive leading coefficient. Since $xs=t_2-t_1$ and $xs$ has odd degree,
it follows that the leading terms of $t_1$ and $t_2$ are equal. Therefore, $z=t_2+xs=2 t_2-t_1$ has even degree with positive leading ooefficient.
Since this property is preserved under sums and products, it follows that every element of $T_{i+1}$ has even degree and positive leading
coefficient.
\end{ex}

%Example \ref{exa} implies that (B1) does not imply (B2) what we already know from Example \ref{exmain}.
Note that $F,G$ from Example \ref{exa} satisfy $(I_2+E_{11})F \in I_2+N_G$. This is in line with Lemma \ref{oneplusq}.
Example \ref{exb} is a modification of Example \ref{exa}.

%Let us show now that (B1) does not imply (B3) for $n=3$. 

\begin{ex} 
\label{exb}
Write $$G=\left[\begin{array}{ccc} 2x & 0 & 0 \\ 0 & x & 0 \\ 0 & 0 & 1 \end{array} \right] \quad \text{ and } \quad
F=\left[\begin{array}{ccc} 1+x & 0 & 0 \\ 0 & 1+x & 0 \\ 0 & 0 & 1 \end{array} \right].$$
We will show that $F$ and $G$ satisfy (B1) but they do not satisfy (B3).
Note that $K_G=\RR$. A short computation shows that for every $x \in \RR$,
$$M_x:=\{r>0 \mid F(x)-r G(x) \succ 0\}=\left\{
\begin{array}{cc} 
(\frac{1+x}{x},1), & x<-1, \\
(0,1), & -1 \le x \le 1, \\
(0, \frac{1+x}{2x}), & x>1.
\end{array} \right.$$
Clearly, $M_x$ is nonempty for every $x \in \RR$, thus $F$ and $G$ satisfy (B1').
If there is a rational function $r$ such that $r(x) \in M_x$ for every $x \in \RR$, 
then we get a contradiction
$$\lim\limits_{x \to +\infty} r(x) \in [0,\frac12] \quad \text{ and } \quad
\lim\limits_{x \to -\infty} r(x)=1.$$
Therefore $F$ and $G$ no not satisfy (B2').
Since $K_G=\RR$, they do not satisfy (B3') either.
\end{ex}

\begin{prop}
\label{containsrat3}
Let $F,G$ be symmetric $2 \times 2$ real univariate matrix polynomials satisfying property (B1). 
Then they also satisfy property (B3).
\end{prop}

\begin{proof}
By Propositions \ref{A3prop} and \ref{containsrat2} it suffices to prove that property (B1') implies the following claim:
\begin{enumerate}
\item[(B3'')] There exists a rational function $r$ such that for some $R > 0$,
$r$ has no singularities in $K_G \setminus (-R,R)$, $r>0$ on $K_G \setminus (-R,R)$
and $F-rG \succ 0$ on $K_G \setminus (-R,R)$.
\end{enumerate}
Recall that we can replace the condition $r>0$ in properties (B1') and (B3'') with $r \ge 0$
by the proof of Lemma \ref{btoa}.

Let us first reduce the claim to the case of diagonal $G$.
If $g_{11} \equiv 0$ and $g_{22} \equiv 0$ then replace $F$ and $G$ with $AFA^T$ and $AGA^T$ where
$$A=\frac{1}{\sqrt{2}} \left[ \begin{array}{cc} 1 & 1 \\ 1 & -1 \end{array} \right]
\quad \text{and} \quad
A^T G A = \left[ \begin{array}{cc} g_{12} & 0 \\ 0 & -g_{12} \end{array} \right].$$
If $g_{11} \not\equiv 0$ (similarly if $g_{22} \not\equiv 0$) then replace $F$ and $G$ with $BFB^T$ and $BGB^T$ where
$$B=\left[ \begin{array}{cc} g_{11} & 0 \\ -g_{12} & g_{11} \end{array} \right]
\quad \text{and} \quad
B^T G B = \left[ \begin{array}{cc} g_{11}^3 & 0 \\ 0 & g_{11}\det G \end{array} \right].$$
In the first case, the claim with old $F,G$ is clearly equivalent to the claim with new $F,G$ for any $R$.
In the second case, however, this equivalence is true only for large $R$.

We will now assume that $G$ is diagonal. If $g_{11} \equiv 0$ and $g_{22} \equiv 0$, then $F \succ 0$ on $\RR$
and we can take $r \equiv 0$ in (B3''). 

If $g_{11} \not\equiv 0$ and $g_{22} \equiv 0$, then $K_G=\RR$
and $f_{22}>0$ on $\RR$. Pick $R>0$ such that $g_{11}$ and $\det F$ have constant signs on
$K_1:=[-\infty,-R]$ and $K_2:=[R,\infty]$. Write $\phi=\frac{\det F}{g_{11} f_{22}}$. 
Clearly, for every $x \in K_1 \cup K_2$, we have that $r \in M_x$ iff
$r \ge 0$ and $\det(F(x)-r G(x))=\det F(x)- r g_{11}(x) f_{22}(x)=g_{11}(x) f_{22}(x)(\phi(x)-r)>0$ iff
either $g_{11}(x)>0$ and $0 \le r< \phi(x)$ or $g_{11}(x)<0$ and $r> \phi(x)$ and $r \ge 0$. 
Note that by (B1'), $g_{11}(x) > 0$ implies  $F(x) \succ 0$ (and so $\phi(x)>0$) for every $x$.
If $g_{11}|_{K_1}>0$ and $g_{11}|_{K_2}>0$, we can take $r \equiv 0$ in (B3'').
If $g_{11}|_{K_1}<0$ and $g_{11}|_{K_2}<0$, we can take $r=1+\phi^2$ in (B3'') as $1+\phi^2>\max\{0,\phi\}$.
If $g_{11}|_{K_1}<0$ and $g_{11}|_{K_2}>0$, then we have two possibilities.
If $\phi|_{K_1}<0$, we can take $r\equiv 0$ in (B3''), and if $\phi|_{K_1} \ge 0$, we can take 
$r=\phi-\eps x (1+x^2)^{-k}$ where $\eps$ and $k$ are such that $\phi(x)> \eps x (1+x^2)^{-k}$ on $K_2$.
(They exist e.g. by Proposition 2.6.2 in \cite{bcr}.)
The case $g_{11}|_{K_1}>0$ and $g_{11}|_{K_2}<0$ is similar.

Finally, we assume that $G$ is diagonal and $g_{11} g_{22} \not\equiv 0$. For each $x \in \RR$, let $r_-=r_-(x)$ and $r_+=r_+(x)$
be the solutions of the equation $0=\det(F-r G)=\det F-(f_{11}g_{22}+f_{22}g_{11})r+g_{11} g_{22}r^2$. Explicitly,
$$r_\pm=\frac{f_{11}g_{22}+f_{22}g_{11} \pm \sqrt{D}}{2 g_{11} g_{22}}$$
where $D=(f_{11}g_{22}+f_{22}g_{11})^2-4 g_{11} g_{22} \det F$. 
Let us show that $D \ge 0$ for every $x \in \RR$. 
If $g_{11} g_{22} \ge 0$ then $D=(f_{11}g_{22}-f_{22}g_{11})^2+4 g_{11} g_{22} f_{12}^2 \ge 0$. If $g_{11} g_{22} < 0$
then the maximum of $\det(F-r G)$ must be positive because $F-rG \succ 0$ for at least one real $r$. Therefore
$D \ge 0$ in this case as well. It follows that $r_\pm(x)$ are both  real for every $x \in \RR$. 
It also follows that $D$ has even degree, so $r_\pm$ have asymptotic expansions of the form $x^\alpha \sum_{i=0}^\infty \frac{c_i}{x^i}$ where $\alpha \in \ZZ$ and $c_i \in \RR$ for all $i$.
We will use several times that
 for every $x \in \RR$, $r_+(x)>r_-(x)$ iff $g_{11}(x)g_{22}(x)>0$. 

Suppose that $x \in \RR$ satisfies $g_{11}(x)g_{22}(x) \ne 0$. We still assume that $x$ satisfies (B1'), i.e. there exists $r_0 \in \RR^{\ge 0}$ such that $F(x)-r_0 G(x) \succ 0$.
Let us show that for every $r \in \RR^{\ge 0}$, the following are equivalent.
\begin{enumerate}
\item $F(x)-rG(x) \succ 0$.
\item $\det(F(x)-r G(x))>0$.
\item One of the following is true:
\begin{enumerate}
\item[(3a)] $g_{11}(x)g_{22}(x)<0$ and $r \in (r_+(x),r_-(x))$.
\item[(3b)] $g_{11}(x)g_{22}(x)>0$ and $r \not\in[r_-(x),r_+(x)]$.
\end{enumerate}
\item One of the following is true:
\begin{enumerate}
\item[(4a')] $g_{11}(x)g_{22}(x)<0$ and $0\le r_+(x) <r<r_-(x)$.
\item[(4a'')] $g_{11}(x)g_{22}(x)<0$ and $r_+(x) <0 \le r<r_-(x)$.
\item[(4b')] $g_{11}(x)>0$ and $g_{22}(x)>0$ and $0\le r <r_-(x)<r_+(x)$.
\item[(4b'')] $g_{11}(x)<0$ and $g_{22}(x)<0$ and $r_+(x)<r$, $0 \le r$.
\end{enumerate}
\item One of the following is true:
\begin{enumerate}
\item[(5a)] $x \in K_G$ and $\det F(x) \le 0$ and $0\le r_+(x) <r<r_-(x)$.
\item[(5b)] $x \in K_G$ and $\det F(x) > 0$ and $0 \le r<r_-(x)$.
\item[(5c)] $x \not\in K_G$ and $r_+(x)<r$, $0 \le r$.
\end{enumerate}
\end{enumerate}

Clearly, (1) implies (2) and (2) is equivalent to (3).
We will show now that (3a) implies (1). A similar argument shows that (3b) implies (1). 
Suppose that $g_{11}(x) g_{22}(x)<0$ and $r \in (r_+(x),r_-(x))$. 
It suffices to show that $f_{11}(x)-r_\pm(x) g_{11}(x) \ge 0$ but not both zero.
Namely, this implies that $f_{11}(x)-r g_{11}(x) >0$ which together with (2) implies (1).
We have that $f_{11}-r_\pm g_{11}=\frac{\Delta \mp \sqrt{D}}{2 g_{22}}$ 
where $\Delta=f_{11}g_{22}-f_{22}g_{11}$. By assumptions, $0 \le D(x)=\Delta(x)^2+4 g_{11}(x) g_{22}(x) f_{12}(x)^2 \le \Delta(x)^2$.
Thus, $\Delta(x) \mp \sqrt{D(x)}$ has the same sign as $\Delta(x)$. On the other 
hand, $\Delta(x)$ has the same sign as $g_{22}(x)$. Namely, write $H$ for the diagonal matrix with entries $g_{22}$ and $-g_{11}$
and note that $\tr GH=0$ and $\Delta=\tr FH$. If $g_{22}(x)>0$ then $H(x) \succ 0$ which implies  that $\Delta(x)>0$ by (B1').
Similarly, if $g_{22}(x)<0$, then $-H(x) \succ 0$, which implies that $-\Delta(x)>0$ by (B1').
In particular, at least one of $\Delta(x) \mp \sqrt{D(x)}$ is nonzero.

Finally, we show that (3), (4) and (5) are equivalent. Property (B1), formula 
$r_+r_-=\frac{\det F}{g_{11}g_{22}}$ and claim (2) of Lemma \ref{mainlem} imply the following four equivalences
for every $x$ such that $g_{11}(x)g_{22}(x) \ne 0$ and every $r \ge 0$. (3a) and $\det F(x) \le 0$ iff (4a') iff (5a).
(3a) and $\det F(x) >0$ iff (4a'') iff (5b) and $g_{11}(x)g_{22}(x)<0$.
(3b) and $G(x) \succ 0$ iff (4b') iff (5b) and $g_{11}(x)g_{22}(x)>0$.
(3b) and $G(x) \prec 0$ iff (4b'') iff (5c).

Pick $R>0$ such that $g_{11},g_{22},r_+,r_-$ have constant signs on $(-\infty,-R]$ and on $[R,\infty)$. 
Write $K_G^+=K_G \cap [R,\infty)$ and $K_G^-=K_G \cap (-\infty,-R]$. 
If (5a) is satisfied for every $x \in K_G^+\cup K_G^-$, then we can take $r =\frac{r_-+r_+}{2}=\frac{f_{11}g_{22}+f_{22}g_{11}}{2 g_{11}g_{22}}$ in (B3'').
If (5b) is satisfied for every $x \in K_G^+\cup K_G^-$, then we can take $r \equiv 0$ in (B3'').
Finally, suppose that $K_G^+$ and $K_G^-$ are nonempty and 
%(5a) is satisfied on $K_G^-$ and (5b) is satisfied on $K_G^+$.
(5a) is satisfied for every $x \in K_G^-$ and (5b) is satisfied for every $x \in K_G^+$ (or vice versa). 
By Proposition 2.6.2. in \cite{bcr} we can pick
$C \in \RR^{> 0}$ and $k \in \NN$ such that $\frac{1}{r_-(x)-r_+(x)} \le C(1+x^2)^k$ on $K_G^-$
and $\frac{1}{r_-(x)} \le C(1+x^2)^k$ on $K_G^+$.
The function $\psi(x):=r_-(x) - \frac{1}{C} (1+x^2)^{-k}$ may not be a rational function 
but it satisfies other requirements for $r$ in (B3'').
%but it satisfies (5a) on $K_G^-$ and (5b) on $K_G^+$.
Therefore, we can take for $r$ in (B3'') an appropriate truncation of the asymptotic series of $\psi$.
\end{proof}

\section{Compact Positivstellensatz for several constraints}
\label{sec6}

It is well-known that Finsler's Lemma fails for several constraints. More precisely, we have the following:

\begin{lemma}
For given $G_1,\ldots,G_m,F \in \sym_n(\RR)$ consider the claims:
\begin{enumerate}
\item[(1)] For every nonzero $v \in \RR^n$ such that $v^T G_1 v \ge 0,\ldots,v^T G_m v \ge 0$, we have that $v^T F v>0$.
\item[(1')] For every nonzero positive semidefinite matrix $B \in \sym_n(\RR)$ such that 
$\tr(G_1 B) \ge 0,\ldots,\tr(G_m B) \ge 0$ we have that $\tr(F B)>0$.
\item[(2)] There exist nonnegative real numbers $r_1,\ldots,r_m$ such that 
$F-r_1 G_1-\ldots-r_m G_m \succ 0$.
\end{enumerate}
Then (2) is equivalent to (1') but it is not always equivalent to (1).
\end{lemma}

\begin{proof}
\begin{comment}
Let $M$ be the convex wedge generated by all positive semidefinite matrices and $G_1,\ldots,G_m$.
Clearly, $M$ has nonempty interior. If $F \not\in M^\circ$, then there exists a functional $\phi$ 
such that $\phi(F) \le 0$ and $\phi(M) \ge 0$ (Eidelheit's Separation Theorem.)
By the Riesz representation theorem, $\phi(\cdot)=\tr(\cdot B)$ for some symmetric matrix $B$. A short computation
shows that $B$ must be nonzero and positive semidefinite (since $\phi$ is positive and nonzero). 
\end{comment}
Clearly, (2) implies (1') and (1') implies (1).
The matrices
$$G_1=\left[ \begin{array}{cc} 1 & 0 \\ 0 & -1 \end{array} \right],\quad 
G_2=\left[ \begin{array}{cc} 0 & -1 \\ -1 & 1 \end{array} \right],\quad 
F=\left[ \begin{array}{cc} 1 & -1 \\ -1 & 0 \end{array} \right]$$
satisfy claim (1) but they do not satisfy claim (1').
Finally, (1') implies (2) by the Separation theorem for convex sets and Riesz representation theorem for linear functionals.
%or Proposition \ref{starprop} below.
\end{proof}

For a given subset $\G$ of $\sn$,
write $N_\G$ for the smallest weak quadratic module containing $\G$
and $T_\G$ for the smallest weak preordering containg $\G$. Write $K_\G$ for the set of all $a \in \RR^d$ for which
there  exists a nonzero positive semidefinite matrix $B \in \sym_n(\RR)$ such that 
$\tr G(a) B \ge 0$ for every $G \in \G$.
We will prove the following:

\begin{theorem}
\label{lastth}
If $\G$ is a subset of $\sn$ such that the set $K_\G$ is compact then for every $F \in \sn$ the following are equivalent:
\begin{enumerate}
\item[(1)] For every $a \in \RR^d$ and every nonzero positive semidefinite matrix $B \in \sym_n(\RR)$ such that 
$\tr G(a) B \ge 0$ for every $G \in \G$, we have that $\tr F(a) B >0$.
\item[(2)] There exists $\eps >0$ such that $F- \eps I_n \in \pos(K_\G) \cdot N_\G$.
\end{enumerate}
\end{theorem}

Since the set $K_\G$ is compact, the preordering $\pos(K_\G)$ is archimedean. It follows that the weak preordering $\pos(K_\G) \cdot N_\G$
is also archimedean. However, $\pos(K_\G) \cdot N_\G$ need not be finitely generated as a weak quadratic module, so
we cannot finish the proof by using Theorem 2.1. from \cite{c1} but we have to use Theorem \ref{arch} below.

\begin{remark}
We do not know whether we can replace the weak preordering $\pos(K_\G) \cdot N_\G$ in Theorem \ref{lastth} with  $T_\G$.
However, we can replace $\pos(K_\G) \cdot N_\G$ with the smaller weak quadratic module $\pos(K_\G) \cdot I_n+N_\G$.
When $\G$ has only one element, we can even replace $\pos(K_\G) \cdot N_\G$ with $N_\G$ by Theorem \ref{mainth2}.
Namely, note that $K_{\{G\}}=K_G$ for every $G \in \sn$.
%In Theorem \ref{lastth} we can replace the weak preordering $\pos(K_\G) \cdot N_G$ with the smaller weak quadratic module $\pos(K_\G) \cdot I_n+N_\G$. 
%If $K_\G$ is a semialgebraic set, i.e. if $K_\G=\bigcup_{i=1}^l K_{S_i}$ for some finite sets $S_1,\ldots,S_l$, then we also replace $\pos(K_\G)$ 
%with $\bigcap_{i=1}^l T_{S_i}$. In all four cases we get an archimedean weak quadratic module.
%We do not know, however, whether we can replace $\pos(K_\G) \cdot N_\G$ 
\end{remark}

Let $B$ be a unital real or complex $\ast$-algebra with center $Z(B)$. Let $B_h$ and $Z(B)_h$ be the hermitian parts of $B$ and $Z(B)$ respectively.
A subset $N$ of $B_h$ is called a \textit{weak quadratic module}  if 
\begin{enumerate}
\item $N+N \subseteq N$, 
\item $a^\ast a \in N$ for every $a \in B$,
\item $c^\ast c N \subseteq N$ for every $c \in Z(B)$.
\end{enumerate}
We say that a weak quadratic module $N$ is \textit{archimedean} if for every $a \in B_h$ there exists $k \in \NN$ such that $k \pm a \in N$.
For every weak quadratic module $N$ write $N^\vee$ for the set of all $N$-\textit{positive states}, i.e. linear functionals  on $B_h$ which satisfy 
$\omega(1)=1$ and $\omega(N) \ge 0$. We say that $\omega \in N^\vee$ is \textit{factorizable} if $\phi(xy)=\phi(x)\phi(y)$
for every $x \in B_h$ and every $y \in Z(B)_h$. The following is a variant of Proposition 1 in \cite{c4} which is an
extension of Vidav-Handelman theory.

\begin{prop}
\label{starprop}
Let $B$ be a unital real or complex $\ast$-algebra and let $N$ be an archimedean weak quadratic module on $N$. 
Then for every $f \in B_h$, the following are equivalent:
\begin{enumerate}
\item $\omega(f)>0$ for every factorizable $N$-positive state $\omega$.
\item $f \in \eps+N$ for some real $\eps>0$.
\end{enumerate}
\end{prop}

\begin{proof}
Clearly, (2) implies (1). Let us show that (1) implies (1') where
\begin{enumerate}
\item[(1')] $\omega(f)>0$ for every $N$-positive state $\omega$.
\end{enumerate}
Since $N$ is archimedean, $1$ is an interior point of $N$ in the finest locally convex topology,
so $V=(N-1) \cap (1-N)$ is a neighbourhood of zero. Banach-Alaoglu Theorem implies that $N^\vee$ 
is compact in the topology of pointwise convergence. 
Since the set of all factorizable 
$N$-positive states is closed in the topology of pointwise convergence, it is also compact,
so (1) implies that there exists $\eps>0$ such that $\omega(f) \ge \eps$ for every factorizable
$\omega \in N^\vee$. Lemma \ref{propex} implies that $\omega(f) \ge \eps$ for every extreme
point $\omega$ of $N^\vee$. By the Krein-Milman theorem, $N^\vee$ is a closed convex hull 
of the set of its extreme points, which gives (1').
Finally, Proposition 1.4. in \cite{c5} shows that (1') implies (2).
\end{proof}

\begin{lemma}
\label{propex}
If $N$ is an archimedean weak quadratic  module in $B$ then every extreme point of $N^\vee$ is factorizable.
% satisfies $$\phi(xy)=\phi(x)\phi(y)$$ for every $x \in B_h$ and every $y \in Z(B)_h$.
\end{lemma}

\begin{proof}
Let $\omega$ be an extreme point of $N^\vee$.
Since $y=\frac14((1+y)^2-(1-y)^2)$ and $\omega$ is additive, we may assume that $y$ is a square.
In particular, $y \in N$ and $yN \subseteq N$.
Since $N$ is archimedean and $\omega$ is homogeneous, we may also assume that $\frac12-y \in N$. 

\medskip

Claim 1: If $\omega(y)=0$, then $\omega(y^2)=0$. 

\medskip

Since $y, 2-y \in N$, it follows that $1-(1-y)^2 = \frac12 \big( y(2-y)^2+(2-y)y^2 \big)\in N$,
thus $\omega((1-y)^2) \le 1$. On the other hand, $\omega((1-y)^2)=\omega(y^2)-2 \omega(y)+1 \ge 1$.
Finally, $\omega((1-y)^2)=1$ implies that $\omega(y^2)=0$.

\medskip

Claim 2: $\omega((1-y)z) \ge 0$ for every $z \in N$.

\medskip

We will modify the proof of Lemma 4.7 in \cite{bss}. We may assume that $1-z \in N$. 
For every $n$ write $q_n(t)$ for the $n$-th Taylor polynomial of $\sqrt{1-t}$ and $p_n(t)=q_n(t)^2-(1-t)$.
We have that $\omega((1-y)z)=\omega(q_n(y)^2 z)+\omega(p_n(y)(1-z))+\omega(p_n(\frac12)-p_n(y))-p_n(\frac12)$.
Since $y^k(1-z) \in N$ and $\frac{1}{2^k}-y^k \in N$ for every $k$ and since $p_n$ has nonnegative coefficients,
it follows that $\omega((1-y)z) \ge -p_n(\frac12)$. Finally, send $n \to \infty$.

\medskip

Case 1: If $\omega(y)=0$, then $\omega(xy)=0$ for every $x \in B_h$.
Namely, by the Cauchy-Schwartz inequality and Claim 1, $\vert \, \omega(xy) \vert^2 \le \omega(x^2)\omega(y^2)=0$.
It follows that $\omega(xy)=\omega(x)\omega(y)$ if $\omega(y)=0$.

\medskip

Case 2 : If $0 < \omega(y)$, then $\omega_1$ and $\omega_2$ defined by
\[
\omega_1(x):= \frac{1}{\omega(y)} \, \omega(xy) \quad \mbox{ and } \quad
\omega_2(x):= \frac{1}{\omega(1-y)} \, \omega(x(1-y))
\]
($x \in B_h$) are $N$-positive states on $B_h$. For $\omega_1$ this is clear from the choice of $y$
and for $\omega_2$ this is exactly Claim 2.

Clearly, $\omega=\omega(y)\omega_1+\omega(1-y)\omega_2$. 
Since $\omega$ is an extreme point of the set of all $N$-positive states on $B_h$, it follows that
$\omega=\omega_1=\omega_2$. In particular,  $\omega(xy)=\omega(x)\omega(y)$.
\end{proof}

Theorem \ref{arch} extends Theorem 2.1 in \cite{c1} from finitely generated weak quadratic modules to all weak quadratic modules.

\begin{theorem}
\label{arch}
Suppose that $\G \subseteq \sn$ is such that $R^2-\Vert x \Vert^2 \in N_\G$ for some real $R$. Then the following are equivalent:
\begin{enumerate}
\item[(1)] For every point $a \in \RR^d$ and every real $0 \ne B \succeq 0$ such that
$\tr G(a)B \ge 0$ for every $G \in \G$, we have that $\tr F(a) B >0$.
\item[(2)] $F - \eps I_n \in N_\G$ for some real $\eps>0$.
\end{enumerate}
\end{theorem}

\begin{proof} Note that every factorizable $\sohs$-positive state $\omega$ on $\sn$ is 
of the form $\omega(H)=\tr H(a) B$ for some point $a \in \RR^d$ and some nonzero
real positive semidefinite matrix $B$. Note also that the assumption
$R^2-\Vert x \Vert^2 \in N_\G$ implies that $N_\G$ is archimedean.
Finally, use Proposition \ref{starprop} with $B=\mn$ and $N=N_\G$.
\end{proof}

Theorem \ref{arch} implies the well-known Scherer-Hol Theorem;
see \cite[Corollary 1]{sh} for the original result
and \cite[Theorem 13]{ks} for the reformulation and extension to infinite $\G$.
An alternative proof and a generalization to infinite dimensions is given
in \cite[Theorem 6]{c4}.

Recall that a weak quadratic module $M$ in $\sn$ is a \textit{quadratic  module}
if $A^T M A \subseteq M$ for every $A \in \mn$. The smallest quadratic module 
containing a given set $\G \subset \sn$ will be denoted by $M_\G$. Clearly,
$M_\G=N_{\G'}$ where $\G'=\{A^T G A \mid G \in \G, A \in \mn\}$.

\begin{cor}
Suppose that $\G \subseteq \sn$ is such that $R^2-\Vert x \Vert^2 \in M_\G$ for some real $R$. Then the following are equivalent:
\begin{enumerate}
\item[(1)] For every point $a \in \RR^d$ such that $G(a) \succeq 0$ for every $G \in \G$, we have that $F(a) \succ 0$.
\item[(2)] $F - \eps I_n \in M_\G$ for some real $\eps>0$.
\end{enumerate}
\end{cor}


\begin{thebibliography}{99}
\bibitem{av} C.-G. Ambrozie, F.-H.   Vasilescu,           
Operator-theoretic Positivstellensätze, 
Z. Anal. Anwend. \textbf{22} (2003), no. 2, 299--314.

\bibitem{bcr}
J. Bochnak, M. Coste, M.-F. Roy, 
\textit{Real algebraic geometry},
Ergebnisse der Mathematik und ihrer Grenzgebiete (3), 36,
Springer-Verlag, Berlin, 1998. x+430 pp, ISBN: 3-540-64663-9.

\bibitem{bss} S. Burgdorf, C. Scheiderer, M. Schweighofer, 
Pure states, nonnegative polynomials and sums of squares, 
Comment. Math. Helv. \textbf{87} (2012), no. 1, 113-140.

\bibitem{c1} J. Cimpri\v c, 
Noncommutative Positivstellensätze for pairs representation-vector. 
Positivity 15, no. 3, 481--495 (2011).

\bibitem{c2}  J. Cimpri\v c, 
Strict Positivstellensätze for matrix polynomials with scalar constraints,
Linear Algebra Appl. \textbf{434}(2011), no. 8, 1879–-1883.

\bibitem{c4} J. Cimpri\v c,  
Archimedean operator-theoretic Positivstellensätze,
J. Funct. Anal. \textbf{260} (2011), no. 10, 3132-–3145.

\bibitem{c5} J. Cimpri\v c, M. Marshall, T. Netzer, 
Closures of quadratic modules,
Israel J. Math. \textbf{183} (2011), 445–-474. 

\bibitem{c6}  J. Cimpri\v c, 
Real algebraic geometry for matrices over commutative rings,
J. Algebra \textbf{359} (2012), 89–-103. 

\bibitem{chmn} J. Cimpri\v c,  J. W. Helton, S. McCullough, C. Nelson, 
A noncommutative real nullstellensatz corresponds to a noncommutative real ideal: algorithms,
Proc. Lond. Math. Soc. (3) \textbf{106} (2013), no. 5, 1060–-1086.

\bibitem{zalar2} J. Cimpri\v c, A. Zalar, 
Moment problems for operator polynomials,
J. Math. Anal. Appl. \textbf{401} (2013), no. 1, 307–-316.

\bibitem{c3} J. Cimpri\v c, 
A Real Nullstellensatz for free modules,
J. Algebra \textbf{396} (2013), 143–-150.

\bibitem{chkmn} J. Cimpri\v c,  J. W. Helton, I. Klep, S. McCullough, C. Nelson, 
On real one-sided ideals in a free algebra,
J. Pure Appl. Algebra \textbf{218} (2014), no. 2, 269–-284.

\bibitem{f} P. Finsler, 
\" Uber das Vorkommen definiter und semidefiniter Formen in Scharen quadratischer Formen,
Comment. Math. Helv. \textbf{9} (1936), no. 1, 188–-192. 

\bibitem{h1} J. W. Helton, S. A. McCullough, M. Putinar, 
Non-negative hereditary polynomials in a free *-algebra,
Math. Z. \textbf{250} (2005), no. 3, 515–-522.

\bibitem{h2} J. W. Helton, S. A. McCullough, M. Putinar, 
Strong majorization in a free *-algebra,
Math. Z. \textbf{255} (2007), no. 3, 579–-596. 

\bibitem{klep}  J. W. Helton, I. Klep, C. S. Nelson, 
Noncommutative polynomials nonnegative on a variety intersect a convex set,
arXiv:1308.0051

\bibitem{ks} I. Klep, M. Schweighofer, 
Pure states, positive matrix polynomials and sums of Hermitian squares,
Indiana Univ. Math. J. \textbf{59} (2010), no. 3, 857–-874.

\bibitem{viet} L\^{e} C\^{o}ng-Trinh, 
Some Positivstellens\" atze for polynomial matrices, 
arXiv:1403.3783

\bibitem{nelson} C. S. Nelson, 
A Real Nullstellensatz for Matrices of Non-Commutative Polynomials,  
arXiv:1305.0799

\bibitem{sh} C.W. Scherer, C.W.J. Hol, 
Matrix sum-of-squares relaxations for robust semi-definite programs, 
Math. Program. \textbf{107} (2006), no. 1–2, Ser. B, 189–-211.
    
\bibitem{sch} K. Schm\" udgen, 
Noncommutative real algebraic geometry—some basic concepts and first ideas,
in \textit{Emerging applications of algebraic geometry}, 325–350, 
IMA Vol. Math. Appl., 149, Springer, New York, 2009. 

\end{thebibliography}
\end{document}